\documentclass[11pt]{amsart}
\usepackage[utf8]{inputenc}
\usepackage{amssymb,latexsym}
\usepackage[mathscr]{eucal}
\usepackage{mathtools}

\theoremstyle{plain}
\newtheorem{theorem}{Theorem}
\newtheorem{corollary}{Corollary}
\newtheorem*{conjecture}{Conjecture}
\newtheorem{lemma}{Lemma}

\theoremstyle{remark}
\newtheorem*{heuristics}{Heuristics}

\newcommand{\refL}[1]{Lemma~\ref{L#1}}

\makeatletter
\@namedef{subjclassname@2020}{%
  \textup{2020} Mathematics Subject Classification}
\makeatother

\newcommand{\gre}{\varepsilon}
\newcommand{\gf}{\varphi}
\newcommand{\cp}{\mathcal{P}}

\begin{document}

\title[Heights of Cyclotomic Polynomials]{A Note on Heights of Cyclotomic Polynomials}

\author[G. Bachman]{Gennady Bachman}
\address{Department of Mathematical Sciences\\ University of Nevada Las Vegas\\
4505 Maryland Parkway, Box 454020\\
Las Vegas, Nevada 89154-4020, USA}
\email{gennady.bachman@unlv.edu}
\author[C. Bao]{Christopher Bao}
\email{baoc1@unlv.nevada.edu}
\author[S. Wu]{Shenlone Wu}
\email{wus19@unlv.nevada.edu}

\thanks{We would like to thank the anonymous referee for helping to improve the exposition in this note.}


\begin{abstract}
We show that for any positive integer $h$, either $h$ or $h+1$ is a height of some cyclotomic polynomial $\Phi_n$, where $n$ is a product of three distinct primes.
\end{abstract}



\maketitle

\section{Introduction}

Let $\Phi_n$ be the $n$th cyclotomic polynomial and write
\[
\Phi_n(x)=\prod_{\substack{1\le a<n\\(a,n)=1}}\bigl(x-e^{2\pi ia/n}\bigr)=\sum_{0\le m\le\gf(n)}a_mx^m \qquad[a_m=a_m(n)],
\]
where $\gf$ is the Euler's totient function. This note is concerned with the coefficients of these polynomials, which the reader will recall are integral, and specifically with the function 
\[
A(n)=\max_m|a_m(n)|.
\]
Sometimes one refers to $A(n)$ as the height of the polynomial $\Phi_n$. The study of coefficients of cyclotomic polynomials began with the curious observation that for small $n$, the nonzero coefficients of $\Phi_n$ seemed to be always $\pm1$. Indeed, $A(n)=1$ for all $1\le n\le104$, but $A(105)=2$. Today, there is a considerable body of literature devoted to the study of these coefficients and especially the function $A(n)$ (see \cite{Bz1,Bz2} and references therein).

Amongst many easy properties of polynomials $\Phi_n$ is the fact that to study their coefficients it suffices to restrict $n$ to be odd and squarefree (see, for example, \cite[Section~1]{Le}). With this restriction in place, it turns out that the number of prime factors of $n$, $\omega(n)$, is an important parameter in this problem. In particular, when $n=p$ is prime or $n=pq$ is a product of two (odd) primes, it is easy to show that $A(p)=A(pq)=1$ (see \cite{Le}). Since $n=105$ is the smallest odd integer with $\omega(n)>2$, the equality $A(n)=1$, for $n\le104$, becomes fully explained. Moreover, we see that from our current perspective the two cases $\omega(n)=1$ or 2 are considered trivial. Parenthetically, the first of these is simply
\[
\Phi_p(x)=\frac{x^p-1}{x-1}=1+x+x^2+\dots+x^{p-1},
\]
but coefficients of polynomials $\Phi_{pq}(x)$ continue to offer interesting and even difficult challenges (see \cite{Le} and \cite{Fo}).

For $\omega(n)>2$, to determine $A(n)$ is far more challenging. For one thing, $A(n)$ is no longer a function of just $\omega(n)$, but rather a proper function of $n$. Furthermore, $A(n)$'s true dependence on $n$ is quite mysterious. In a letter to Landau, Schur showed that $A(n)$ is not bounded (see \cite{Leh} or \cite {Le}). Later Suzuki \cite{Su} pointed out that Schur's argument shows that every integer occurs as a coefficient of some polynomial $\Phi_n$. It is then natural to ask if every positive integer occurs as a height of some polynomial $\Phi_n$. In other words, for any positive integer $h$, is the equation
\begin{equation}\label{a1}
A(n)=h
\end{equation}
is soluble in $n$?

This question was tackled by Kosyak, Moree, Sofos and Zhang in their recent paper \cite{KMSZ}. Earlier work of Moree and Ro\c{s}u \cite{MR} generates a range of integers $I_h$ depending on $h$, such that if $I_h$ contains a prime number $p$ then we can give a solution to \eqref{a1} using this $p$. Roughly, this approach requires for there to be a prime $p$ in the interval
\begin{equation}\label{a2}
2h-1-\sqrt{2h-1}<p\le2h-1.
\end{equation}
This requirement reminds one of a famous conjecture of Legendre that there is always a prime between consecutive squares. Despite tremendous recent progress in the theory of distribution of primes, the requirement \eqref{a2} is well out of reach. But recent advances by Heath-Brown \cite{HB} allow Kosyak et al to conclude that \eqref{a2}, whence the equation \eqref{a1}, are satisfied for almost all $h$. More precisely, they conclude that as $x$ goes to infinity, the number of ``bad'' $h\le x$ is $O_\gre(x^{3/5+\gre})$.

Using a different approach, we compliment the work of Kosyak et al, by showing that given any two consecutive positive integers, at least one of them must be a height of some cyclotomic polynomial. Our approach does not require any deep state of the art distribution of primes results, such as \eqref{a2}, and when it comes to that, it rests entirely on Dirichlet's theorem for primes in arithmetic progressions---the absolute minimum and indispensable tool in this business. In fact, our construction of solutions to \eqref{a1} can be achieved with such a sparse subset of primes $\cp$ that $\cp$ has hardly any primes in it at all. The precise statement of our results and observations are given in Sections 2 and 3, after a discussion of relevant background material.

\section{Background and Results}

It seems that the first ``theoretical'' evaluation of $A(n)$ for $\omega(n)>2$ (as opposed to explicitly calculating coefficients of $\Phi_n$ for particular values of $n$) was the following result of Bachman \cite{Ba1}.

\begin{lemma}\label{L1}
Let $p$ be an arbitrary odd prime and let primes $q>p$ and $r>pq$ satisfy the congruences
\[
q\equiv2\pmod p\quad\text{and}\quad r\equiv\frac{pq\pm1}2\pmod{pq}.
\]
Then $A(pqr)=\frac{p+1}2$.
\end{lemma}

We say that $\Phi_n$ is a ternary cyclotomic polynomial if $n=pqr$ is a product of three distinct odd primes. Another result of Bachman \cite{Ba2} gave the first targeted evaluation of $A(n)$, an evaluation designed to produce a particular height $h$. He showed that if $p$ is an arbitrary prime and primes $q$ and $r$ satisfy
\[
q\equiv-1\pmod p\quad\text{and}\quad r\equiv1\pmod{pq},
\]
then $A(pqr)=1$. This result was greatly generalized by Kaplan \cite{Ka1} who showed that the condition $r\equiv\pm1\pmod{pq}$ alone was sufficient to guarantee that $A(pqr)=1$. This was followed by Moree and Ro\c{s}u's generalization of \refL{1} \cite{MR}.

\begin{lemma}\label{L2}
Let $p$ be an odd prime and let $x_p$ be the largest of the two zeros of the polynomial $4x^2+2x+3-p$. Then for any integer $h$ in the interval
\[
\frac{p+1}2\le h\le\frac{p+1}2+x_p,
\]
the equation $A(pqr)=h$ is soluble in primes $q$ and $r$. Moreover, one can certainly take $q$ to be any sufficiently large prime in a certain residue class $q\equiv a_h\pmod p$, and then take $r$ to be any sufficiently large prime with respect to $pq$ in a certain residue class $r\equiv b_h\pmod{pq}$.
\end{lemma}

This pretty much sums up what is known about evaluating $A(pqr)$. Virtually nothing is known about evaluations of $A(n)$ with $\omega(n)>3$ (but see \cite{Ka2}). It should now be clear that the basic height problem of solving $A(n)=h$, should be reinterpreted as the problem of solving this equation with $\omega(n)=3$. The true crux of the matter is the question of solubility of
\[
\text{$A(n)=h$ in $n$, with $\omega(n)=k$ fixed}.
\]
For $k>3$ this appears to be far more challenging, and one expects the difficulty to go up with $k$. A particularly appealing special case is $h=1$. After all, the study of coefficients of cyclotomic polynomials begun with the question of whether $A(n)=1$ for all $n$.

Returning to the simplest case of our problem, $\omega(n)=3$, let us reiterate that the results of Kosyak et al discussed in the introduction were for this case and based on \refL{2}. In the present paper we compliment these results as follows.

\begin{theorem}
Let $h$ be a given positive integer and put $p'=2h-1$. If a triple of odd primes $p,q,r$ satisfies the congruences
\begin{equation}\label{b1}
q\equiv2\pmod{p'},\quad r\equiv\frac{p'q+1}2\pmod{p'q},\quad p\equiv p'\pmod{qr},
\end{equation}
then
\begin{equation}\label{b2}
A(pqr)=h\text{ or }h+1.
\end{equation}
\end{theorem}

A few remarks are in order here. First, the existence of primes specified in \eqref{b1} is guaranteed by Dirichlet's theorem for primes in arithmetic progressions (see, for example, \cite[Chapter~7]{Ap}). Second, $h=1$ gives $p'=1$ and $p\equiv1\pmod{qr}$. In this case we always get $A(pqr)=1$, by the above mentioned result of Kaplan. Third, if $p'=2h-1$ happens to be prime, then $A(p'qr)=h$, by \refL{1} (and taking $p=p'$ in \eqref{b1} is one of the options). Finally, the conclusion \eqref{b2} is fully valid in the sense that both cases of equality do occur. This is so even when $p'$ is prime: there are examples where
\[
p\equiv p'\pmod{qr}\implies A(pqr)=A(p'qr)=h,
\]
and examples where
\[
p\equiv p'\pmod{qr},\ p>p' \implies A(pqr)=A(p'qr)+1=h+1.
\]
We shall briefly return to this point in the next section, where we will see that it leads to some rather interesting questions.

Unlike the requirement \eqref{a2}, our solutions to \eqref{b2} are not built around a set of prime numbers contained in a short interval determined by the target $h$. Instead, taking $p>p'$ in \eqref{b1}, even if $p'$ is a prime, we see that our solution to \eqref{b2} works with $\min(p,q,r)$ arbitrarily large. In fact, we can make this point much stronger by showing that one can solve \eqref{b2} in primes restricted to a sparse subset $\cp\subset\mathbb{P}$ of the set of prime numbers. To that end, let us introduce the counting function $P(x)$, which gives the number of elements of $\cp$ not exceeding $x$. We shall give the following example.

\begin{corollary}
There is a subset of primes $\cp$ with $P(x)<\log x$, such that, for every $h$, one can solve \eqref{b2} in arbitrarily large primes $p,q,r\in\cp$.
\end{corollary}

This leaves little room for doubt that every positive integer is a height of some $\Phi_{pqr}$. In the next section, after reviewing the concept of inclusion-exclusion polynomials, we give another heuristic argument which makes this even clearer. In fact, we make the following conjecture.

\begin{conjecture}
For an odd prime $p$, let
\[
M(p)\coloneqq\max_{q,r}A(pqr).
\]
Then for every $p$ and $1\le h\le M(p)$, the equation
\[
A(pqr)=h
\]
is soluble in primes $q$ and $r$.
\end{conjecture}

\section{Ternary Inclusion-Exclusion Polynomials}

The class of inclusion-exclusion polynomials was defined in \cite{Ba3}. These polynomials may be thought of as a tool for studying coefficients of cyclotomic polynomials and other divisors of polynomials $x^n-1$. In the present paper we are concerned only with ternary cyclotomic polynomials and, correspondingly, we will reference only ternary inclusion-exclusion polynomials defined below. In this section we let the triple of letters $p,q,r$ denote any three integers greater than 2 and relatively prime in pairs (\emph{they need not be primes}). For any such triple of parameters $\{p,q,r\}$, the ternary inclusion-exclusion polynomial $Q_{\{p,q,r\}}(x)$ is defined by the expression
\begin{equation}\label{c1}
Q_{\{p,q,r\}}(x)\coloneqq\frac{(x^{pqr}-1)(x^p-1)(x^q-1)(x^r-1)}{(x^{pq}-1)(x^{qr}-1)(x^{rp}-1)(x-1)}.
\end{equation}
This polynomial has degree
\[
\gf(p,q,r)\coloneqq(p-1)(q-1)(r-1)
\]
and we write
\[
Q_{\{p,q,r\}}(x)=\sum_{0\le m\le\gf(p,q,r)} a_mx^m \qquad [a_m=a_m(\{p,q,r\})]
\]
and
\[
A(p,q,r)=A(\{p,q,r\})\coloneqq\max_m |a_m(\{p,q,r\})|.
\]
Restricting parameters $p,q,r$ to be (distinct) prime numbers yields cyclotomic polynomials $\{\Phi_{pqr}\}$ as a subclass of $\{Q_{\{p,q,r\}}\}$. For these and other basic properties of inclusion-exclusion polynomials the reader is referred to \cite{Ba3}. Representation \eqref{c1} is the usual starting point for the investigation of coefficients of $\Phi_{pqr}$/$Q_{\{p,q,r\}}$.

\begin{lemma}\label{L3}
Suppose that $r\equiv s\pmod{pq}$. Then we have
\begin{gather}
r,s>\max(p,q) \implies A(p,q,r)=A(p,q,s), \label{c2}\\
r>\max(p,q)>s\ge1 \implies A(p,q,s)\le A(p,q,r)\le A(p,q,s)+1. \label{c3}
\end{gather}
When $s=1$ or $2$ in \eqref{c3}, by convention, $A(p,q,1)=0$ and $A(p,q,2)=1$.
\end{lemma}

The first of these implications was established by Kaplan \cite{Ka1}. Even though he worked in the setting of cyclotomic polynomials, the result holds more generally (as stated) \cite{Ba3}. To avoid repeating this comment again, we will simply state true facts about inclusion-exclusion polynomials even though they may have been formally proven for cyclotomic polynomials only. The more challenging implication \eqref{c3} was proven by Bachman and Moree in \cite{BM}. The result is sharp and both of the two possible equalities do occur.

Lemmas 1 and 2 also fully generalize to the setting of inclusion-exclusion polynomials. We will only use the first of these below, so let us state it explicitly and in a stronger form permitted by \eqref{c2}.

\begin{lemma}\label{L4}
Let $p$ be an odd number, let $q$ be of the form $q=2+(2k+1)p$, for some $k\ge0$, and let $r$ be of the form $r=\frac{pq+1}2+lpq$, for some $l\ge0$. Then $A(p,q,r)=\frac{p+1}2$.
\end{lemma}

\begin{proof}[Proof of Theorem 1]
The proof of Theorem 1 is now immediate. Given $h$, let $p'=2h-1$ and let primes $p,\ q$ and $r$ satisfy \eqref{b1}. \refL{4} gives $A(p',q,r)=h$, from which \eqref{b2} follows by an application of \eqref{c3} in the form
\[
A(q,r,p')\le A(q,r,p)\le A(q,r,p')+1.
\]
\end{proof}

Our argument raises several interesting points and gives a strong heuristic support for our basic claim.

\begin{heuristics} Let us consider the possibility that some $h\ge2$ is never a height of a cyclotomic polynomial. By \refL{4}, there are inclusion-exclusion polynomials with $A(2h-1,q,r)=h$, for infinitely many choices of prime numbers $q$ and $r$. (Parenthetically, using \refL{2} adds many more options for generating inclusion-exclusion polynomials of height $h$.) And for each such polynomial there must always be a jump in height of the form
\begin{equation}\label{c4}
A(2h-1+qr,q,r)=h+1,
\end{equation}
for otherwise, by \eqref{c2}, there would be a prime $p$ with $A(p,q,r)=h$, contradicting our assumptions. Applying the same reasoning with $h-1$ in place of $h$ shows that no such jump ever happens:
\[
A(2h-3,q,r)=h-1\implies A(2h-3+qr,q,r)=h-1.
\]
Since there is no perceptible difference between $h$ and $h-1$, \emph{always} jumping \eqref{c4} for one and \emph{never} for the other is surely not reasonable.
\end{heuristics}

Another interesting question is whether jumps \eqref{c4} can be used to sequentially generate all of $\mathbb N$. More precisely, define a recurrence relation
\[
(p_{i+1},q_{i+1},r_{i+1})=(q_i,r_i,p_i+q_ir_i) \qquad[i\ge1],
\]
for any initial triple of relatively prime in pairs integers $2<p_1<q_1<r_1$, and consider the sequence of heights $\{A(p_i,q_i,r_i)\}_{i=1}^\infty$ thus generated. By \eqref{c3},
\[
A(p_{i+1},q_{i+1},r_{i+1})-A(p_i,q_i,r_i) = 0\text{ or } 1,
\]
whence the set of generated heights $h$ is either all integers starting with $A(p_1,q_1,r_1)$, or all integers in some interval $A(p_1,q_1,r_1)\le h\le H$. In the latter case $A(p_i,q_i,r_i)=H$ for all $i\ge i_0$. It is natural to ask whether there is a choice of $(p_1,q_1,r_1)$ such that 
\[
\{A(p_i,q_i,r_i)\}=\mathbb N.
\]
Notice that if any of these height sequences are bounded, that means that there are initial conditions $A(p_1,q_1,r_1)$ resulting in the constant sequence
\[
A(p_i,q_i,r_i)\equiv A(p_1,q_1,r_1).
\]
It is not at all clear whether any such sequences exist.

Given the relative novelty of inclusion-exclusion polynomials it might be beneficial to recast the preceding question in the more familiar (but more restrictive) setting of cyclotomic polynomials. We do this by way of an example. Extensive list of heights of cyclotomic polynomials have been compiled by Arnold and Monagan \cite{AM}. Consulting their list we find that the ``first'' ternary $\Phi_{pqr}$ of height 1 correspond to the triple of primes 3, 7, 11, that is, $A(3\cdot7\cdot11)=1$. The smallest ``next'' prime $\equiv3\pmod{7\cdot11}$ is 157 and $A(7\cdot11\cdot157)=2$ (according to the list). Continuing this scheme generates the sequence
\[
A(3\cdot7\cdot11)=1,\, A(7\cdot11\cdot157)=2,\, A(11\cdot157\cdot3461),\,\dots
\]
of heights of the corresponding cyclotomic polynomials, and the question is whether every positive integer occurs as an element of this sequence.

\section{Proof of the Corollary}

Recall that $\cp$ denotes a subset of prime numbers and $P(x)$ denotes the number of elements of $\cp$ not exceeding $x$. We will also write $a\pmod m$ to denote the arithmetic progression $a, a+m, a+2m,\dots$, and $|\cp\cap a\pmod m|=\infty$ indicates that infinitely many elements of $\cp$ are in $a\pmod m$.

\begin{lemma}\label{L5}
Suppose that a set of primes $\cp$ satisfies the following properties.
\begin{enumerate}
\item[(P1)] If $m$ is an arbitrary odd number, then $|\cp\cap2\pmod m|=\infty$.
\item[(P2)] If $m$ is an arbitrary odd number, then $|\cp\cap\frac{m+1}2\pmod m|=\infty$.
\item[(P3)] For an arbitrary $a\in\mathbb N$, if distinct $q,r\in\cp$ are sufficiently large in terms of $a$, say $q,r\ge C_a$, then $|\cp\cap a\pmod{qr}|=\infty$.
\end{enumerate}
Then for each positive integer $h$, there are arbitrarily large triples of primes $p,q,r\in\cp$ for which the conclusion \eqref{b2} of Theorem 1 holds.
\end{lemma}

\begin{proof}
Given a target height $h$, put $p'=2h-1$ and, one by one, select primes $q,\ r$, and $p$ in $\cp$ satisfying $q\equiv2\pmod{p'},\ q>C_{p'}$, $r\equiv\frac{p'q+1}2\pmod{p'q},\ r>C_{p'}$, and $p\equiv p'\pmod{qr}$. These elements of $\cp$ satisfy the conditions \eqref{b1} and the conclusion \eqref{b2} follows.
\end{proof}

To complete the proof of Corollary 1, we will show that there is a set of primes $\cp$ satisfying (P1)-(P3) and having $P(x)<\log x$. Let us introduce the notation
\[
\overline{2n+1}\coloneqq1\cdot3\cdot\ldots\cdot(2n+1)\qquad[n\ge1].
\]
We begin by specifying two sequences of primes $\{q_n\}$ and $\{r_n\}$ as follows. We thake $q_1=5$ and for each $n\ge2$ starting with $n=2$, we fix a prime number $q_n$ satisfying
\[
q_n\equiv2\pmod{\overline{2n+1}}\quad\text{and}\quad q_n>q_{n-1}.
\]
Similarly, we take $r_1=5$ and for each $n\ge2$, we fix a prime $r_n$ satisfying
\begin{equation}\label{d1}
r_n\equiv\frac{\overline{2n+1}+1}2\pmod{\overline{2n+1}}\quad\text{and}\quad r_n>\overline{2n+1}, r_{n-1}.
\end{equation}
Note that both of the specified selections are possible by Dirichlet's theorem. In the case of \eqref{d1}, we note that $\gcd(\frac{\overline{2n+1}+1}2,\overline{2n+1})=1$. Now, for every odd number $m$, $m\mid\overline{2n+1}$ for all sufficiently large $n$. For each such $n$, we have $q_n\equiv2\pmod m$ and
\[
r_n\equiv\frac{\overline{2n+1}+1}2\equiv\frac{m+1}2\pmod m.
\]
It follows that properties (P1) and (P2) are satisfied by any set $\cp$ containing primes $\{q_n\}$ and $\{r_n\}$ as subsets.

To estimate the number of elements in $\{q_n\}\cup\{r_n\}$ not exceeding $x$, we simply observe that
\[
q_n,r_n>\overline{2n+1}>n!>81^n\qquad[n\ge n_0],
\]
whence $q_n,r_n>x$ for  $n\ge\frac14\log x$ . Therefore, we have
\begin{equation}\label{d2}
\#\{\,q_n,r_n : q_n, r_n\le x\,\}<\frac12\log x\qquad[x\ge x_0].
\end{equation}

Next we address (P3) of \refL{5}. To do this we construct, for each $a$, a sequence of primes $\{p_n(a)\}$ as follows. In our construction we use the notation $\pi_j$ to denote the $j$th prime number ($\pi_1=2, \pi_2=3, \dots$) and put
\[
P(i,j)\coloneqq\pi_i\pi_{i+1}\dots\pi_j\qquad[i<j].
\]
Given $a$, let $\pi_k$ be the smallest odd prime $>a^a$ ($a=1\mapsto\pi_2=3$) and let:
\[
p_1(a) : \quad\text{smallest prime} \equiv a\pmod{\pi_k}\quad \text{and}\quad >\pi_k^3,
\]
and then successively, for $n\ge1$,
\[
p_{n+1}(a) : \quad\text{smallest prime} \equiv a\pmod{P(k,k+n)}\quad \text{and}\quad >p_n(a)^3,
\]
The resulting sequence $\{p_n(a)\}$ has the property that for any two primes $q,r>a^a$, and so $q,r\ge\pi_k$, we have $p_n(a)\equiv a\pmod{qr}$ for all sufficiently large $n$. It follows that any set $\cp$ containing $\cup_a\{p_n(a)\}$ is guaranteed to satisfy (P3) of \refL5.

To estimate the number of elements in $\cup_a\{p_n(a)\}$ not exceeding $x$, we note that for all sufficiently large $x$ we have
\[
a\ge\frac12\frac{\log x}{\log\log x}\implies p_1(a)>a^{3a}>x,
\]
and
\[
n\ge\log\log x\implies p_n(a)>x,
\]
since $p_n(a)\ge p_1(a)^{3^{n-1}}>3^{3^n}$. Therefore, we have
\begin{equation}\label{d3}
\#\{\,p_n(a) : p_n(a)\le x\,\}<\frac12\log x\qquad[x\ge x_0].
\end{equation}

Finally, forming the union
\[
\cp=\{q_n\}\cup\{r_n\}\cup\bigl(\cup_a\{p_n(a)\}\bigr)
\]
creates a set which satisfies (P1)-(P3) of \refL{5}, as was already mentioned, and, by \eqref{d2} and \eqref{d3}, satisfies the desired inequality $P(x)<\log x$, for sufficiently large $x$. Plainly, by removing enough of small elements of $\cp$ we can force this inequality to hold for all $x$.

\end{document}